\documentclass[11pt,reqno]{amsart}
\usepackage[T1]{fontenc}
\usepackage{amsmath,amssymb,mathtools,booktabs,enumitem}
\usepackage[expansion=false]{microtype}
\usepackage{cite}
\usepackage[hidelinks]{hyperref}
\numberwithin{equation}{section}
\theoremstyle{plain}
\newtheorem{theorem}{Theorem}[section]
\newtheorem{lemma}[theorem]{Lemma}
\newtheorem{proposition}[theorem]{Proposition}
\newtheorem{corollary}[theorem]{Corollary}
\theoremstyle{definition}
\newtheorem{definition}[theorem]{Definition}
\newtheorem{remark}[theorem]{Remark}
\newcommand{\MI}{\mathsf{MI}}
\newcommand{\IB}{\mathsf{IB}}
\renewcommand{\AA}{\mathsf{A}}
\newcommand{\RR}{\mathsf{R}}
\newcommand{\JJ}{\mathsf{J}}
\newcommand{\BB}{\mathsf{B}}
\newcommand{\RB}{\mathsf{RB}}
\newcommand{\BR}{\mathsf{BR}}
\newcommand{\KK}{\mathsf{K}_0}
\newcommand{\TT}{\mathbf{T}}
\newcommand{\Id}{\operatorname{Id}}
\newcommand{\Var}{\operatorname{Var}}
\newcommand{\Mod}{\operatorname{Mod}}
\newcommand{\ct}{\operatorname{c}}
\newcommand{\first}{\ell}
\newcommand{\last}{r}
\newcommand{\rev}{\operatorname{rev}}

\newcommand{\sr}{\mathrm{SR}}
\newcommand{\G}{\mathcal{G}}
\newcommand{\Dset}{\mathcal{D}}
\newcommand{\Iso}{\mathcal{I}}

\hypersetup{pdftitle={The Finite Basis Problem for Multiplicatively Idempotent Semirings},pdfauthor={Lili Wang, Aifa Wang}}

\title[Multiplicatively idempotent semirings]{The Finite Basis Problem for Multiplicatively Idempotent Semirings}
\author[A. Wang]{Aifa Wang}
\address{School of Mathematical Sciences, Chongqing University of Technology, Chongqing 400054, P.R. China}
\email{wangaf@cqut.edu.cn}
\author[L. Wang]{Lili Wang}
\address{School of Mathematical Sciences, Chongqing University of Technology, Chongqing 400054, P.R. China}
\email{wllaf@cqut.edu.cn}
\thanks{This work was supported by the Science and Technology Research Program of Chongqing Municipal
Education Commission (KJZD-K2024011102) and the Chongqing Natural Science Foundation Innovation
and Development Joint Fund (Municipal Education Commission) (CSTB2025NSCQ-LZX0067).}
\date{}

\begin{document}
\begin{abstract}
The variety of semirings satisfying $x^2\approx x$ is hereditarily finitely based in the signature consisting of addition and multiplication, without designated constants. Multiplication may be noncommutative, and the semirings need not be finite. A subdirect decomposition reduces the proof to the subvarieties defined by $3x\approx x$ and $3x\approx2x$. In the latter case, we express each polynomial identity using a common doubled part and isolated word occurrences. Four fixed identities align the first- and last-occurrence words of these occurrences. Induction on their content extracts protected identities. The associated word identities define a band variety, to which the finite basis theorem for bands applies.
\end{abstract}
\maketitle

\section{Introduction}
We ask whether every multiplicatively idempotent semiring has a finite identity basis in the signature $\{+,\cdot\}$. A semiring is an algebra $(S,+,\cdot)$ in which addition is associative and commutative, multiplication is associative, and multiplication distributes over addition from both sides. We call a semiring \emph{multiplicatively idempotent} if it satisfies $x^2\approx x$, and write $\MI$ for the variety of all such semirings. Neither an additive identity nor a multiplicative identity is part of the signature, and we do not assume commutative multiplication. For a positive integer $n$, the notation $nx$ denotes a sum of $n$ copies of $x$.

An algebra is \emph{finitely based} if its identities follow from a finite set of identities. A variety is \emph{hereditarily finitely based} if each of its subvarieties is finitely based. Additively idempotent semirings can be nonfinitely based even when their multiplicative reducts are finitely based; see Jackson, Ren and Zhao~\cite{JRZ}. Ren, Yue and Zhou~\cite{RYZ} investigate inherent nonfinite basability in the additively idempotent setting.

For multiplicatively idempotent semirings, Vechtomov and Petrov~\cite{VP} proved that $\MI$ is locally finite and that, when multiplication is commutative, its members decompose subdirectly into semirings satisfying $3x\approx x$ and $3x\approx2x$. We establish this decomposition for noncommutative multiplication as well.

In the commutative setting, Chajda and L\"anger~\cite{CLSI,CLV} studied multiplicatively idempotent semirings with designated constants $0$ and $1$. They obtained a finite basis and polynomial normal forms for the variety generated by the three-element semiring $T_3$~\cite{CLV}. Its $\{+,\cdot\}$-reduct occurs below as an obstruction to the identity $2(x+y)\approx x+y$. We work with this reduct in the constant-free signature.

Wang and Shao~\cite[Corollary~3.4]{WS} proved hereditary finite basability for the semiring variety determined, for each integer $n\ge2$, by
\[
 x^n\approx x,\qquad x+(2n-2)xyx\approx x.
\]

At $n=2$, this theorem applies to the subvariety of $\MI$ satisfying $x+2xyx\approx x$. The reduct of $T_3$ used below fails this absorption identity. This failure shows that the theorem does not cover all of $\MI$.

Every variety of bands is finitely based~\cite{GP}. The variety of semirings in which both operations are idempotent has $78$ subvarieties, all finitely based~\cite{GPZ,Pastijn,Polak}. These results do not by themselves control the undoubled summands of a semiring polynomial. An arbitrary band can occur as a multiplicative subsemigroup after adjoining an absorbing element and making addition constant. The multiplicative reducts of these semirings need not satisfy a regular-band identity.

We prove the following theorem.

\begin{theorem}\label{thm:main}
The variety $\MI$ is hereditarily finitely based in the signature $\{+,\cdot\}$.
\end{theorem}

\begin{corollary}\label{cor:algebras}

Every multiplicatively idempotent semiring has a finite identity basis. Every finite multiplicatively idempotent semiring is finitely based.

\end{corollary}

The defining axioms give a finite basis for $\MI$ itself. Theorem~\ref{thm:main} asserts finite basability of every subvariety, a conclusion that does not follow from local finiteness.

For the subvariety satisfying $3x\approx2x$, we express each polynomial identity using a common doubled part and isolated word occurrences. We apply the theory of doubly idempotent semirings to the doubled part; the isolated word occurrences retain the original band structure. Four fixed identities align their first- and last-occurrence words. Once these words agree, recursive extraction gives an identity of the form
\[
 u+2u\omega u\approx v+2v\omega v,
\]
where $\omega$ is a fresh variable. The associated word identities form a fully invariant band congruence. The finite basis theorem for bands applies to the resulting band variety.

\section{Preliminaries}\label{sec:prelim}

We follow the terminology of universal algebra in~\cite{BS}. For an algebra $S$, let $\Var(S)$ be the variety generated by $S$. For a variety $\mathsf W$, let $\Id(\mathsf W)$ be its identity theory. We take bases for semiring subvarieties relative to the following finite set $\sr$:
\begin{equation}\label{eq:sr}
\begin{gathered}
 x+y\approx y+x,\qquad (x+y)+z\approx x+(y+z),\qquad (xy)z\approx x(yz),\\
 x(y+z)\approx xy+xz,\qquad (x+y)z\approx xz+yz.
\end{gathered}
\end{equation}
We have $\MI=\Mod(\sr,x^2\approx x)$. We write identities with $\approx$ and use equality for literal words and for elements under a fixed evaluation.

Fix a countably infinite alphabet. A \emph{word} is a nonempty multiplicative word, and $\ct(w)$ is its content. Let $\first(w)$ and $\last(w)$ denote the \emph{first-occurrence word} and \emph{last-occurrence word} of $w$: retain, respectively, the first and the last occurrence of each variable in its original left-to-right order. A polynomial is a finite nonempty sum of word occurrences. We count occurrences with multiplicity, so $u+u$ has two occurrences. An occurrence is \emph{supported in} a set of variables when its content is a subset of that set. Every semiring term expands into such a polynomial. Missing contexts and omitted sums are notation only; we never substitute an empty word or an empty sum for a variable.

Write $\BB$ for the variety of bands, that is, idempotent semigroups, and $B_X$ for the free band on a finite nonempty alphabet $X$. Every subvariety of $\BB$ has a finite basis~\cite{GP}. The structure theorem for bands~\cite{McLean} gives the following deletion rule.

\begin{lemma}\label{lem:deletion}
Let $p,b,q$ be nonempty words. If $\ct(p)=\ct(q)$ and $\ct(b)\subseteq\ct(p)$, then
\[
 \BB\models pbq\approx pq.
\]
The identity $aba\approx a$ holds in every band whenever $\ct(b)\subseteq\ct(a)$.
\end{lemma}
\begin{proof}
Every band is a semilattice of rectangular bands. Under an arbitrary evaluation, let $P,B,Q$ be the values of $p,b,q$. The content assumptions imply that $P,PB,Q$ belong to the same rectangular component. Since a rectangular band satisfies $xyz\approx xz$,
\[
 PBQ=P(PB)Q=PQ.
\]
Taking $p=q=a$ proves the final assertion.
\end{proof}

For a word $w$ involving at least two variables, let $s(w)$ be its longest prefix whose content is a proper subset of $\ct(w)$, and let $e(w)$ be its longest suffix with proper content. These are actual factors of $w$, possibly with repeated variables. If $a$ is the last letter of $\first(w)$ and $b$ is the first letter of $\last(w)$, then
\begin{equation}\label{eq:proper-content}
 \ct(s(w))=\ct(w)\setminus\{a\},\qquad
 \ct(e(w))=\ct(w)\setminus\{b\}.
\end{equation}

Vechtomov and Petrov~\cite[p.~54 of the Russian original]{VP} observed that $\MI$ is locally finite. The proof below bounds the number of normal forms used later.

\begin{lemma}\label{lem:localfinite}
Every finitely generated free band is finite. The variety $\MI$ is locally finite.
\end{lemma}
\begin{proof}
Consider a word $w$ on at most $n$ variables. Let $p$ and $q$ be its shortest full-content prefix and suffix. If these factors are separated, Lemma~\ref{lem:deletion} removes the intervening factor. If they overlap, writing $p=ah$, $q=hb$ gives $pq=ahhb\approx ahb=w$. Both cases give $w\approx pq$. Writing $p=s(w)a$ and $q=be(w)$ and proceeding by induction shows that every word is equivalent to one of length at most
\[
 L_1=1,\qquad L_n=2L_{n-1}+2=3\cdot2^{n-1}-2.
\]
This proves that $B_X$ is finite.

In $\MI$ we have
\begin{equation}\label{eq:fourtwo}
 2x=(x+x)^2=4x.
\end{equation}
Equation~\eqref{eq:fourtwo} reduces every positive coefficient in a polynomial to $1$, $2$ or $3$. Replacing each word by a representative from $B_X$ leaves at most $4^{|B_X|}-1$ nonempty polynomials with these coefficients, which proves local finiteness.
\end{proof}

This bound counts polynomial representatives; distinct coefficient vectors may define the same term function.

\begin{lemma}\label{lem:retractions}
For $S\in\MI$, the maps $d(a)=2a$ and $r(a)=3a$ are idempotent endomorphisms. Their images are their respective fixed-point sets, and
\[
 d(S)\in\IB:=\MI\cap\Mod(2x\approx x).
\]
For every $\mathsf W\subseteq\MI$, the class of images $d(S)$, $S\in\mathsf W$, is precisely $\mathsf W\cap\IB$.
\end{lemma}
\begin{proof}
Additivity follows from commutativity of addition. Equation~\eqref{eq:fourtwo} gives
\[
 (2a)(2b)=4ab=2ab,\qquad (3a)(3b)=9ab=3ab,
\]
and $d^2=d$, $r^2=r$. The image of an idempotent map is its fixed-point set. Every $d(S)$ is a homomorphic image of $S$ and belongs to $\IB$. For $S\in\mathsf W\cap\IB$, we have $d(S)=S$.
\end{proof}

An algebra in $\IB$ is called \emph{doubly idempotent}. We use three results:
\begin{enumerate}[label=(\roman*),leftmargin=2.4em]
\item every subvariety of $\IB$ is finitely based;
\item every multiplicative reduct of a member of $\IB$ is a regular band;
\item in the variety $\RB$ of regular bands,
\begin{equation}\label{eq:rbword}
 u\approx v\quad\Longleftrightarrow\quad
 \first(u)=\first(v)\text{ and }\last(u)=\last(v).
\end{equation}
\end{enumerate}
Here $\RB=\Mod(x^2\approx x,xyxzx\approx xyzx)$. Statement (i) follows from~\cite{GPZ,Pastijn}; see~\cite[Result~4 and Theorem~7]{Polak}. Pastijn and Zhao~\cite[Theorem~2.3]{PZ} proved (ii), and Pol\'ak~\cite[Result~1]{Polak} recalls the word criterion in (iii). We include a verification of (ii).

\begin{lemma}\label{lem:regular-double}
Every member of $\MI$ satisfies $2xyxzx\approx2xyzx$.
\end{lemma}
\begin{proof}
In a doubly idempotent semiring, put $a\le b$ when $a+b=b$. Multiplication is monotone, and $(u+v)^2=u+v$ implies $uv\le u+v$. Multiplying this inequality on the left and right gives
\begin{equation}\label{eq:order-step}
 rus=rvs=t\quad\Longrightarrow\quad ruvs\le t.
\end{equation}
Let $p=xyxzx$, $q=xyzx$, and $t_0=xyzyxzx$. The following applications of~\eqref{eq:order-step} use only band identities to simplify their entries:
\[
\begin{array}{c|c|c|c|c|c}
 r&u&v&s&rus=rvs&ruvs\\\hline
 xyx&y&z&xzx&p&q\\
 xyz&y&x&zx&q&t_0\\
 xy&x&z&yzyxzx&t_0&p
\end{array}
\]
The three rows give $q\le p$, $t_0\le q$, and $p\le t_0$, proving $p=q$. Apply this result to the doubled image in Lemma~\ref{lem:retractions}.
\end{proof}

\section{Structural reductions}\label{sec:reduction}

Put
\[
 \RR=\MI\cap\Mod(3x\approx x),\qquad
 \AA=\MI\cap\Mod(3x\approx2x).
\]
Vechtomov and Petrov~\cite[Proposition~3.1]{VP} proved this decomposition for commutative multiplication. We extend it to arbitrary members of $\MI$ and to each subvariety.

\begin{proposition}\label{prop:decompose}
Every $S\in\MI$ has a subdirect embedding
\begin{equation}\label{eq:decompose}
 S\hookrightarrow r(S)\times S/\mathcal H_+,
 \qquad a\longmapsto(3a,[a]_{\mathcal H_+}),
\end{equation}
where $r(S)\in\RR$ and $S/\mathcal H_+\in\AA$. For every $\mathsf W\subseteq\MI$,
\begin{equation}\label{eq:join-ra}
 \mathsf W=(\mathsf W\cap\RR)\vee(\mathsf W\cap\AA).
\end{equation}
\end{proposition}
\begin{proof}
Define $a\preceq b$ if $a=b$ or $b=a+u$ for some $u\in S$, and let $\mathcal H_+$ be the symmetric part of this preorder. This is the Green relation of the commutative additive semigroup. Every additive translation and every left or right multiplicative translation preserves $\preceq$, by distributivity. This translation invariance makes $\mathcal H_+$ a semiring congruence. As $3a=2a+a$ and $2a=3a+a$, its quotient satisfies $3x\approx2x$.

We claim that $\mathcal H_+\cap\ker r$ is the equality relation. Suppose $a\mathcal H_+b$ and $a\ne b$. Write $b=a+u$ and $a=b+v$. Iterating $a=a+u+v$ gives $a=a+2u+2v$. By $4u=2u$,
\[
 a+2u=a+4u+2v=a+2u+2v=a.
\]
With $c=3u$, we obtain $3c=c$, $b=a+c$, and $a=b+c$. If $3a=3b$, then $3a+c=3a$. Adding $a$ gives $e+c=e$, where $e=2a$. Put $p=ac$. Right multiplication by $c$ yields $2p+c=2p$, and $3p=p$ gives $p+c=p$. Expanding the square gives
\[
 b=(a+c)^2=a+p+ca+c=a+p+ca,
\]
which gives $b+c=b$. Since $a=b+c$, this contradicts $a\ne b$. This proves the claim and the embedding~\eqref{eq:decompose}. Its two coordinates are surjective homomorphisms. The factors are homomorphic images of $S$; applying the decomposition to every member of $\mathsf W$ proves~\eqref{eq:join-ra}.
\end{proof}

\begin{lemma}\label{lem:compact}
Let $\mathsf V$ be finitely based and let $\mathsf V_1,\mathsf V_2$ be subvarieties, each defined relative to $\mathsf V$ by finitely many identities. Suppose that
\[
 \mathsf W=(\mathsf W\cap\mathsf V_1)\vee(\mathsf W\cap\mathsf V_2)
 \quad\text{for every }\mathsf W\subseteq\mathsf V.
\]
If $\mathsf V_1$ and $\mathsf V_2$ are hereditarily finitely based, then so is $\mathsf V$.
\end{lemma}
\begin{proof}
Fix $\mathsf W\subseteq\mathsf V$ and finite bases $\Sigma_i$ for $\mathsf W\cap\mathsf V_i$. Let $\Theta_i$ define $\mathsf V_i$ relative to $\mathsf V$. Since
\[
 \Id(\mathsf W)\cup\Theta_i\models\Sigma_i,
\]
finiteness of equational proofs gives a finite $\Delta_i\subseteq\Id(\mathsf W)$ with $\Id(\mathsf V)\cup\Delta_i\cup\Theta_i\models\Sigma_i$. Define $\mathsf Z$ by a finite basis of $\mathsf V$ together with $\Delta_1\cup\Delta_2$. These identities give $\mathsf W\subseteq\mathsf Z\subseteq\mathsf V$ and $\mathsf Z\cap\mathsf V_i=\mathsf W\cap\mathsf V_i$ for $i=1,2$. The assumed decomposition, applied to both $\mathsf Z$ and $\mathsf W$, gives $\mathsf Z=\mathsf W$.
\end{proof}

\subsection{The branch \texorpdfstring{$3x\approx x$}{3x = x}}
Let $\BR$ be the variety of Boolean rings in the signature $\{+,\cdot\}$. The identities $\sr$, $x^2\approx x$, $2x\approx2y$, and $3x\approx x$ define this variety. The common value of $2x$ is the zero element of the ring, although it is not a constant symbol.

For $S\in\RR$ and $e\in d(S)$, the set
\[
 G_e=\{a\in S:2a=e\}
\]
is a subsemiring. Its addition is a group of exponent at most two, with identity element $e$, since $a+e=3a=a$ and $a+a=e$. Distributivity and multiplicative idempotence make it a Boolean ring. Its multiplication is commutative: expanding $(a+b)^2=a+b$ and cancelling in the additive group gives $ab+ba=e$ and $ab=ba$.

\begin{lemma}\label{lem:fibre-content}
Let $p,q\in S\in\RR$ and $e\in d(S)$ satisfy $pqp=p$, $qpq=q$, and $2p+e=2q+e=e$. Then $p+e=q+e$.
\end{lemma}
\begin{proof}
Both $\widehat p=p+e$ and $\widehat q=q+e$ belong to $G_e$, since $2e=e$ and $2p+e=2q+e=e$. Expand $(p+e)(q+e)(p+e)$. The term containing no $e$ is $pqp=p$.

Let $t$ be any remaining product term. The factor $e$ gives $2t=t$. Replacing any factor $p$ or $q$ in this term by $2p$ or $2q$ leaves its value unchanged. All its factors belong to $d(S)$ and are at most $e$ in the additive order, because $2p+e=2q+e=e$. Monotonicity of multiplication on $d(S)$ gives $t\le e^3=e$. Since the expansion contains the term $e^3=e$ itself, the sum of all terms containing $e$ is exactly $e$, and
\[
 \widehat p\widehat q\widehat p=p+e=\widehat p.
\]
The same argument gives $\widehat q\widehat p\widehat q=\widehat q$. In the Boolean ring $G_e$, commutativity and idempotence of multiplication make both left sides equal $\widehat p\widehat q$. This proves $\widehat p=\widehat q$.
\end{proof}

For a polynomial $P$, define $\pi(P)$ to be the set of nonempty contents occurring an odd number of times among its word occurrences. In Boolean rings, a word depends only on its content, and distinct square-free polynomials over the two-element Boolean ring define distinct functions. This uniqueness gives the criterion
\begin{equation}\label{eq:parity}
 \BR\models P\approx Q\quad\Longleftrightarrow\quad\pi(P)=\pi(Q).
\end{equation}
The empty parity set records a zero value; it does not denote an empty sum term.

\begin{proposition}\label{prop:r}
The variety $\RR$ is hereditarily finitely based. More precisely, if $S\in\RR\setminus\IB$, then
\begin{equation}\label{eq:r-join}
 \Var(S)=\Var(d(S))\vee\BR.
\end{equation}
\end{proposition}
\begin{proof}
Choose $a\ne2a$. The set $\{2a,a\}$ is a two-element Boolean ring reduct. This subsemiring gives $\BR\subseteq\Var(S)$. The algebra $d(S)$ is a homomorphic image of $S$. To prove the reverse inclusion, suppose $P\approx Q$ holds in both $d(S)$ and $\BR$. Under any evaluation in $S$, put $e=2P=2Q$. Each word value $w$ occurring in either polynomial satisfies $2w+e=e$. Replacing it by $w+e$ leaves the value of its polynomial unchanged, because $P+e=P$ and $Q+e=Q$. All new summands lie in $G_e$. Words with the same content agree there by Lemmas~\ref{lem:deletion} and~\ref{lem:fibre-content}; their sum depends only on content parity. Equation~\eqref{eq:parity} gives $P=Q$ and proves~\eqref{eq:r-join}.

Every subvariety of $\RR$ is either a subvariety $\mathsf U$ of $\IB$ or $\mathsf U\vee\BR$ for such a $\mathsf U$. Subvarieties of the first kind are finitely based by the finite basis theorem for $\IB$~\cite{GPZ,Pastijn}. For the second kind, choose a finite relative basis $\Sigma_{\mathsf U}$ in $\IB$. The identities
\[
 \sr,\quad x^2\approx x,\quad3x\approx x,\quad
 \{2P\approx2Q:(P\approx Q)\in\Sigma_{\mathsf U}\}
\]
define $\mathsf U\vee\BR$: their models have doubled image in $\mathsf U$, and~\eqref{eq:r-join} gives the required inclusion. The reverse inclusion follows by evaluation.
\end{proof}

\subsection{The branch in which every sum is idempotent}
Put
\[
 \JJ=\AA\cap\Mod(2(x+y)\approx x+y),\qquad
 \KK=\MI\cap\Mod(x+y\approx z+t).
\]
The latter is the variety of constant-addition band semirings. The common sum is a multiplicative absorbing element, by distributivity. Every term containing addition has that same value.

\begin{lemma}\label{lem:kzero}
The variety $\KK$ is hereditarily finitely based.
\end{lemma}
\begin{proof}
Let $\mathsf C\subseteq\KK$ be nontrivial, and let $\mathsf U$ be the band variety generated by its multiplicative reducts. Choose a finite band basis for $\mathsf U$. Together with the axioms of $\KK$ it defines $\mathsf C$: identities between words are precisely those of $\mathsf U$, and identities between terms containing addition hold because both sides have the common sum as their value. No identity between a word and a term containing addition can hold in a nontrivial member, since assigning all variables a nonabsorbing element separates them. The identity $x\approx y$ defines the trivial variety.
\end{proof}

\begin{proposition}\label{prop:j}
The variety $\JJ$ is hereditarily finitely based.
\end{proposition}
\begin{proof}
For $S\in\JJ$, let $D=d(S)$. Every sum belongs to $D$, and $D$ is a multiplicative ideal since $2(es)=(2e)s=es$ and $2(se)=s(2e)=se$ for $e\in D$. The relation $\rho_D$ which collapses $D$ to one class and leaves other elements separate is a semiring congruence. Its quotient has constant addition. The map
\[
 S\longrightarrow d(S)\times S/\rho_D,\qquad s\longmapsto(2s,[s]_{\rho_D})
\]
is a subdirect embedding: the first coordinate separates any two elements in the collapsed class, and the second separates any pair not both in $D$. Applying this embedding to every $\mathsf W\subseteq\JJ$ gives
\[
 \mathsf W=(\mathsf W\cap\IB)\vee(\mathsf W\cap\KK).
\]
Apply Lemma~\ref{lem:compact}, Lemma~\ref{lem:kzero}, and the hereditary finite basability of $\IB$.
\end{proof}

\section{Absorption and polynomial normal forms}\label{sec:normal}

For the subvarieties of $\AA$ outside $\JJ$, we need normal forms for polynomial identities. The additive identity $3x\approx2x$ lets us remove absorbed summands one at a time without cancellation.

\begin{lemma}\label{lem:absorption}
The following assertions hold in $\AA$.
\begin{enumerate}[label=(\roman*),leftmargin=2.4em]
\item If $s+a+b=s$, then $s+a=s$ and $s+b=s$.
\item The identity $x+y+xy\approx x+y$ holds.
\item For each element $s$, the set $I_s=\{a:s+a=s\}$ is closed under addition and multiplication. If $a\in I_s$, then $sa,as\in I_s$.
\end{enumerate}
\end{lemma}
\begin{proof}
If $s+a+b=s$, iteration gives $s+2a+2b=s$, and
\[
 s+a=s+3a+2b=s+2a+2b=s.
\]
This proves (i). Expanding $(x+y)^2=x+y$ gives
\[
 x+y+xy+yx=x+y,
\]
and (i) gives (ii). For (iii), associativity gives additive closure. If $a,b\in I_s$, then
\[
 s=(s+a)(s+b)=s+sb+as+ab,
\]
and (i) gives $ab\in I_s$. Multiplying $s+a=s$ on either side by $s$ proves the remaining assertions.
\end{proof}

Substituting $y=x$ in assertion (ii) gives $3x\approx2x$. This substitution proves that the absorption identity characterizes $\AA$ inside $\MI$.

\subsection{A three-element obstruction}
Let $\TT$ be the three-element semiring with the following tables:
\begin{equation}\label{eq:T}
\begin{array}{c|ccc}
 +&e&s&f\\\hline
 e&e&e&e\\
 s&e&e&s\\
 f&e&s&f
\end{array}
\qquad\qquad
\begin{array}{c|ccc}
 \cdot&e&s&f\\\hline
 e&e&e&f\\
 s&e&s&f\\
 f&f&f&f
\end{array}
\end{equation}
The symbols $e,s,f$ are element names. The tables give $\TT\in\AA\setminus\JJ$. Chajda and L\"anger introduced the semiring $T_3$ in~\cite[Definition~4.1]{CLSI} and studied it in~\cite{CLV}; $\TT$ is its $\{+,\cdot\}$-reduct. In the notation $T_3=\{0,a,1\}$ of~\cite{CLV}, the following map is an isomorphism:
\[
 e\longmapsto a,\qquad s\longmapsto1,\qquad f\longmapsto0.
\]
The symbols in these images name elements of the cited algebra; they do not add constants to the present signature.

\begin{proposition}\label{prop:forcedT}
Every $S\in\AA\setminus\JJ$ contains a subsemiring isomorphic to $\TT$. A subvariety of $\AA$ which does not contain $\TT$ is finitely based.
\end{proposition}
\begin{proof}
Choose $a,b\in S$ with $s=a+b\ne2s$, and put
\[
 e=2s,\qquad f=2ab+2ba+2aba+2bab.
\]
By Lemma~\ref{lem:absorption}, $s$ absorbs $ab$ and $ba$ and their products $aba$ and $bab$. These absorptions give $s+f=s$. Expansion using band identities gives
\[
\begin{aligned}
 sf&=2(ab+aba+aba+ab+bab+ba+ba+bab)=f,\\
 fs&=2(aba+ab+ba+bab+aba+ab+ba+bab)=f.
\end{aligned}
\]
We have $es=se=e$, $ef=fe=f$, $s+e=e$, $e+f=e$, and $2f=f$. Since $s\ne2s$, it is distinct from $e$ and $f$. If $e=f$, the equations $s+f=s$ and $s+e=e$ imply $s=e$, a contradiction. The resulting tables on $\{e,s,f\}$ are~\eqref{eq:T}. The final assertion follows from Proposition~\ref{prop:j}.
\end{proof}

In $\TT$, taking $x=y=s$ gives $x+2xyx=3s=e\ne s$. This evaluation proves that the absorption identity required in~\cite{WS} at $n=2$ does not follow from multiplicative idempotence.

For a polynomial $P$ on a finite alphabet $X$ and a subset $D\subseteq X$, define
\begin{equation}\label{eq:nu}
 \nu_P(D)=\min\bigl(2,\#\{\text{occurrences }w\text{ in }P:\ct(w)\subseteq D\}\bigr).
\end{equation}
An occurrence $u$ in $P$ is \emph{isolated} if $\nu_P(\ct(u))=1$. We call its content an isolated content of $P$ and denote the set of these contents by $\Iso(P)$. To distinguish reduced polynomials, Chajda and L\"anger~\cite[Theorem~4.1]{CLV} used assignments to $0$ and $1$. The evaluations below adapt those assignments to the constant-free signature. The truncated count records how many occurrences have content in a given set, up to a maximum of two, for a possibly noncommutative polynomial.

\begin{lemma}\label{lem:count}
If $\TT\models P\approx Q$, then $\nu_P(D)=\nu_Q(D)$ for every $D\subseteq\ct(P)\cup\ct(Q)$. In this case $\Iso(P)=\Iso(Q)$, and for each common isolated content $D$, the unique occurrences supported in $D$ have content exactly $D$.
\end{lemma}
\begin{proof}
Assign the variables in $D$ the value $s$ and all remaining variables the value $f$. A word takes the value $s$ exactly when its content is contained in $D$; otherwise its value is $f$. A polynomial takes the value $f$, $s$, or $e$ according as the number of such occurrences is $0$, $1$, or at least $2$. This proves equality of the functions $\nu$.

If $u$ is isolated in $P$ with content $D$, there is exactly one occurrence $v$ of $Q$ supported in $D$. Put $E=\ct(v)$. The equality $\nu_Q(E)=1$ gives an occurrence of $P$ supported in $E$. This occurrence must be $u$, since $E\subseteq D$, giving $D\subseteq E$. The two inclusions give $E=D$. The symmetric argument completes the proof.
\end{proof}

\subsection{Lifting doubled absorption}

\begin{lemma}\label{lem:lifting}
Let $P$ be a polynomial and $q$ a word. Suppose $P$ has two distinct occurrences $u,v$ with $\ct(u),\ct(v)\subseteq\ct(q)$. Relative to the axioms of $\AA$, the identity
\begin{equation}\label{eq:doubleabs}
 2P+2q\approx2P
\end{equation}
implies $P+q\approx P$.
\end{lemma}
\begin{proof}
Put $p=uv$. The two actual occurrences and Lemma~\ref{lem:absorption}(ii) give $P+p\approx P$. The absorption set of $P$ contains $2p$, $pP$, $Pp$, and their doubles. Multiply~\eqref{eq:doubleabs} on the left and on the right by $p$ to get
\[
 2pP+2pq\approx2pP,\qquad 2Pp+2qp\approx2Pp.
\]
Adding $P$ shows that $P$ absorbs $2pq$ and $2qp$, and Lemma~\ref{lem:absorption}(i) separates one copy of each. Multiplicative closure of $I_P$ shows that $P$ absorbs
\[
 (qp)(pq)=qpq\approx q,
\]
where the last identity follows from Lemma~\ref{lem:deletion}. This equational derivation applies in every subvariety of $\AA$.
\end{proof}

The two occurrences in Lemma~\ref{lem:lifting} may be equal as words, but both copies must already be present. We cannot duplicate a single isolated occurrence.

Let $\mathsf V\subseteq\IB$ and let $P$ be a polynomial on a fixed finite alphabet $X$. Choose one word representative for each element of $B_X$ and put
\begin{equation}\label{eq:Dset}
 \Dset_{\mathsf V}(P)=\{q\in B_X:\nu_P(\ct(q))=2,\quad
                    \mathsf V\models P+q\approx P\}.
\end{equation}
Content and the absorption condition are independent of the representatives.

\begin{proposition}\label{prop:normal}
Suppose $\mathsf Z\subseteq\AA$ and $d(S)\in\mathsf V$ for every $S\in\mathsf Z$. If $P$ has at least two occurrences, then $\Dset_{\mathsf V}(P)$ is nonempty and, in $\mathsf Z$,
\begin{equation}\label{eq:normal}
 P\approx C_P+\sum_{D\in\Iso(P)}u_D,
 \qquad C_P=\sum_{q\in\Dset_{\mathsf V}(P)}2q,
\end{equation}
where $u_D$ is the unique isolated occurrence with content $D$. If there are no isolated occurrences, we write only $C_P$.

If $P\approx Q$ holds in a variety $\mathsf W\subseteq\AA$ containing $\TT$, and $\mathsf V=\mathsf W\cap\IB$, then, using the same ambient alphabet, the two sets in~\eqref{eq:Dset} coincide. We may use the same doubled polynomial $C$ for $P$ and $Q$.
\end{proposition}
\begin{proof}
For $q\in\Dset_{\mathsf V}(P)$, the doubled image condition implies $\mathsf Z\models2P+2q\approx2P$. Lemma~\ref{lem:lifting} gives $P+q\approx P$ and $P+2q\approx P$. Each nonisolated occurrence $w$ of $P$ belongs to $\Dset_{\mathsf V}(P)$: its count is $2$, and $P$ absorbs it in every additively idempotent semiring. We add all the doubled summands and remove each nonisolated original occurrence using $w+2w\approx2w$. This proves~\eqref{eq:normal}.

To prove nonemptiness, take the product of all the word occurrences of $P$. In a doubly idempotent semiring, $P$ absorbs every occurrence and their product by Lemma~\ref{lem:absorption}. The product contains the content of every occurrence; its count is $2$. It lies in~\eqref{eq:Dset}.

For the final assertion, Lemma~\ref{lem:count} gives $\nu_P=\nu_Q$ and common isolated contents. Since $\mathsf V\models P\approx Q$, the absorption conditions in~\eqref{eq:Dset} agree.
\end{proof}

Adding the doubled words in~\eqref{eq:normal} preserves $\nu_P$ and all isolated contents: if $q\in\Dset_{\mathsf V}(P)$ and $\ct(q)\subseteq E$, then $\nu_P(E)=2$ already.

\subsection{Absorption of words with full content}

We use Pol\'ak's description of $\IB$~\cite{Polak} to test absorption of full-content words. For a polynomial $R$ on $X$, a letter $a\in X$, and $E\subseteq X$, put
\begin{equation}\label{eq:rho}
 \rho_R(a,E)=1
\end{equation}
if some occurrence $w$ of $R$ contains $a$ and all letters preceding the first $a$ in $w$ lie in $E$. This condition is monotone in $E$. For $i\in\{0,1,2\}$, we call a repetition-free word $a_1\cdots a_k$ \emph{left admissible at level $i$} for $R$ under the following conditions:
\begin{enumerate}[label=(\arabic*),start=0,leftmargin=2.4em]
\item there is no restriction;
\item if $k>0$, the letter $a_1$ is the first letter of some occurrence of $R$;
\item $\rho_R(a_j,\{a_1,\ldots,a_{j-1}\})=1$ for every $j$.
\end{enumerate}
The empty sequence is admissible at each level. Reversing all the words defines right admissibility. The empty sequence belongs to this combinatorial definition and does not introduce a semiring term.

\begin{lemma}\label{lem:licensing}
For each $\mathsf V\subseteq\IB$, there are $\lambda,\mu\in\{0,1,2\}$ with the following property. If $\ct(R)=\ct(g)=X$, then
\begin{equation}\label{eq:license}
 \mathsf V\models R+g\approx R
\end{equation}
if and only if $\first(g)$ is left admissible at level $\lambda$ and $\rev(\last(g))$ is left admissible at level $\mu$ for the reversed polynomial $\rev(R)$.
\end{lemma}
\begin{proof}
By~\cite[Result 4]{Polak}, every subvariety of $\IB$ is generated by a subset of the eleven algebras
\[
 L,R,D,M,M^0,B,C,L^0,R^0,B^0,C^0
\]
in that paper. The algebras $L$ and $R$ have two-element chain addition and left-zero or right-zero multiplication; $D$ is the two-element distributive lattice; and $M$ has both operations equal to chain join. The algebra $B$ is a two-element left-zero band with an identity element adjoined, ordered as $a<1<b$, and $C$ is its right-zero dual. Addition is join. The superscript $0$ adjoins an element which is an additive identity and a multiplicative zero. These descriptions specify elements, without adding constants to the signature.

Pol\'ak's absorption tests~\cite[Result 3]{Polak} say that $L$ tests the first letter and $R$ the last letter. The algebra $B$ tests, for every set $Y$ of deleted letters, whether the first letter of $g$ after deleting $Y$ is the first letter of some occurrence of $R$ after the same deletion; $C$ gives the reverse test. For $D,M,M^0$, the condition $\ct(g)=\ct(R)=X$ satisfies the absorption tests. The test for an extension by $0$ restricts $R$ to occurrences supported in $\ct(g)$, which here retains all occurrences. The algebras $L,L^0$ impose left level $1$; $B,B^0$ impose left level $2$; their duals impose the corresponding right levels.

To identify the test for $B$ with level $2$, write $\first(g)=a_1\cdots a_n$. Necessity follows by deleting $\{a_1,\ldots,a_{j-1}\}$ for each $j$. For sufficiency, let $Y\subsetneq X$ be any proper subset and let $a_j$ be the first letter of $g$ outside $Y$. Its predecessors in $\first(g)$ lie in $Y$. Monotonicity in~\eqref{eq:rho} supplies an occurrence of $R$ whose first surviving letter is $a_j$. Deleting all of $X$ gives equal empty words. Take the strongest left and right requirements among the chosen generators. These determine $\lambda,\mu$ and prove the lemma, including the trivial variety, for which both levels can be $0$.
\end{proof}

\begin{lemma}\label{lem:admissible-concat}
At any of the three levels, every prefix of an admissible repetition-free word is admissible. If $p$ and $q$ are left admissible, then $\first(pq)$ is left admissible. Each $\first(w)$, for an occurrence $w$ of $R$, is left admissible at every level. The analogous statements hold on the right.
\end{lemma}
\begin{proof}
At level $2$, a letter of $q$ absent from $p$ has at least the same predecessor set in $\first(pq)$ as it had in $q$. Monotonicity of~\eqref{eq:rho} preserves its admissibility. Prefix closure and the assertion for an occurrence follow from the definition. At level $0$ there is no restriction; at level $1$ the only condition concerns the initial letter.
\end{proof}

\section{Protected identities and recursive extraction}\label{sec:extraction}

The normal forms leave isolated word occurrences that the doubled part does not determine. To compare them, we introduce protected identities. The next results apply to arbitrary subvarieties of $\MI$ and do not require the additive identity $3x\approx2x$.

\begin{definition}\label{def:protected}
For $\mathsf W\subseteq\MI$, define a relation $\G_{\mathsf W}$ on nonempty words by
\begin{equation}\label{eq:protected}
 u\equiv_{\G_{\mathsf W}}v
 \quad\Longleftrightarrow\quad
 \mathsf W\models u+2u\omega u\approx v+2v\omega v,
\end{equation}
where $\omega$ is a variable not occurring in $u,v$. We call the displayed semiring identity the \emph{protected identity} associated with $u\approx v$.
\end{definition}

\begin{lemma}\label{lem:G}
The relation $\G_{\mathsf W}$ is a fully invariant semigroup congruence containing the band identities. It is the identity theory of a band variety and has a finite relative band basis.
\end{lemma}
\begin{proof}
Reflexivity and symmetry follow from equality. For transitivity, use the same fresh variable for the intermediate word. Before applying a nonempty word substitution to~\eqref{eq:protected}, rename $\omega$ to avoid all substitution images; the substituted identity has the same protected form.

For nonempty contexts $p,q$, replace $\omega$ in~\eqref{eq:protected} by $q\omega p$ and multiply on the left by $p$ and on the right by $q$. The result is
\[
 puq+2(puq)\omega(puq)\approx pvq+2(pvq)\omega(pvq).
\]
If only the left context is present, use $\omega\mapsto\omega p$; if only the right context is present, use $\omega\mapsto q\omega$. None of these substitutions uses an empty word. They establish full invariance and compatibility with contexts. Substituting either side of a band identity into a word or its doubled summand preserves equality, so every band identity belongs to the relation. The finite basis theorem for bands gives the finite relative basis.
\end{proof}

The relation $\G_{\mathsf W}$ is a theory of \emph{word} identities; $+$ occurs only in its definition. We call it the \emph{protected band theory} of $\mathsf W$.

\begin{lemma}\label{lem:single-extraction}
Suppose $\mathsf W\models P\approx Q$ and that each side has precisely one occurrence supported in a nonempty set $D$. Let these occurrences be $u,v$, each with content $D$, and suppose
\begin{equation}\label{eq:same-profile}
 \first(u)=\first(v),\qquad\last(u)=\last(v).
\end{equation}
If $|D|\ge2$, put $L=s(u)$, $R=e(u)$, and $t=LvR$. Then $u\equiv_{\G_{\mathsf W}}t$.
\end{lemma}
\begin{proof}
Let $h=u\omega u$. In $P\approx Q$, replace every variable outside $D$ by $2h$ and keep all variables of $D$ fixed. Multiply the result on the left by $L$ and on the right by $R$, and add $2h$ to both sides. The isolated occurrences become $LuR\approx u$ and $LvR=t$, since $L$ and $R$ are an actual prefix and suffix of $u$.

Every other occurrence contains a variable outside $D$. After distribution, it gives $2^jLw'R$ for some $j\ge1$, where $w'$ contains a full copy of $h$. By~\eqref{eq:fourtwo} its coefficient is $2$. The prefix $L$ contains all letters of $D$ except the last letter of $\first(u)$. That last letter appears before $\omega$ in the first copy of $h$. Dually, $R$ fixes the last-occurrence word. These prefix and suffix conditions give
\[
 \first(Lw'R)=\first(u)\omega,\qquad
 \last(Lw'R)=\omega\last(u).
\]
By~\eqref{eq:rbword} and Lemma~\ref{lem:regular-double}, each such doubled word equals $2h$. The added $2h$ absorbs all these copies, and the resulting identity is
\[
 u+2h\approx t+2h.
\]
The word $t=LvR$ has the same first- and last-occurrence words as $u$. The prefix $L$ supplies their common proper first-occurrence prefix, $R$ supplies their common proper last-occurrence suffix, and $v$ supplies the missing letters in the required order. These matching orders give $2t\omega t\approx2h$, which turns the last displayed identity into the protected identity for $u,t$.
\end{proof}

\begin{theorem}[Recursive extraction]\label{thm:extraction}
Let $\mathsf W\subseteq\MI$ and $\mathsf W\models P\approx Q$. Suppose each side has exactly one occurrence supported in $D$, these occurrences $u,v$ have content $D$, and~\eqref{eq:same-profile} holds. Then
\begin{equation}\label{eq:extraction-result}
 \mathsf W\models u+2u\omega u\approx v+2v\omega v.
\end{equation}
\end{theorem}
\begin{proof}
We induct on $k=|D|$. If $k=1$, both words are positive powers of the same variable and agree in every band. Suppose $k\ge2$ and assume the result for smaller contents. Write
\[
 \first(u)=\first(v)=a_1\cdots a_k,\qquad
 \last(u)=\last(v)=b_1\cdots b_k.
\]
Set $L=s(u)$, $R=e(u)$, $S=s(v)$, and $F=e(v)$. Lemma~\ref{lem:single-extraction} gives
\begin{equation}\label{eq:uLvR}
 u\equiv_{\G_{\mathsf W}}LvR.
\end{equation}
We show that $Lv\equiv_{\G_{\mathsf W}}v$ and $vR\equiv_{\G_{\mathsf W}}v$.

Choose $d\in D\setminus\{a_k\}$ and let $\alpha$ replace $a_k$ by $d$, fixing all other variables, including those outside $D$. The transformed isolated words have common content $E=D\setminus\{a_k\}$. Every other occurrence still contains a variable outside $D$ and outside $E$. To obtain the first-occurrence word after this identification, apply $\alpha$ to the former first-occurrence word and retain the first occurrence of each resulting letter. The analogous procedure, retaining last occurrences, gives the last-occurrence word. Both procedures preserve the equalities in~\eqref{eq:same-profile}. Since $|E|=k-1$, the induction hypothesis applies and gives
\begin{equation}\label{eq:compressed-left}
 u\alpha\equiv_{\G_{\mathsf W}}v\alpha.
\end{equation}
The substitution fixes $L$ and $S$, both of content $E$. The word $u\alpha$ begins with $L$, the word $v\alpha$ begins with $S$, and both use only letters of $E$. Lemma~\ref{lem:deletion} gives
\[
 (u\alpha)S\approx LS,\qquad (v\alpha)S\approx S.
\]
Multiplying~\eqref{eq:compressed-left} on the right by $S$ within the congruence yields $LS\equiv_{\G_{\mathsf W}}S$. Write $v=Sv_1$, where $v_1$ is its actual nonempty remaining suffix, and append $v_1$. We obtain
\begin{equation}\label{eq:left-stability}
 Lv\equiv_{\G_{\mathsf W}}v.
\end{equation}

For the dual argument, identify $b_1$ with any other letter of $D$, fixing all outside variables; call this substitution $\beta$. The induction hypothesis gives $u\beta\equiv_{\G_{\mathsf W}}v\beta$. The substitution fixes $R,F$, which have content $D\setminus\{b_1\}$. Left multiplication by $F$ and Lemma~\ref{lem:deletion} give
\[
 F(u\beta)\approx FR,\qquad F(v\beta)\approx F.
\]
We obtain $FR\equiv_{\G_{\mathsf W}}F$. Prepend the actual nonempty prefix $v_0$ in $v=v_0F$ to get
\begin{equation}\label{eq:right-stability}
 vR\equiv_{\G_{\mathsf W}}v.
\end{equation}
Combining~\eqref{eq:uLvR}--\eqref{eq:right-stability} in the congruence gives
\[
 u\equiv_{\G_{\mathsf W}}LvR
   \equiv_{\G_{\mathsf W}}vR
   \equiv_{\G_{\mathsf W}}v.
\]
\end{proof}

The induction preserves isolation by identifying variables only within the isolated content. It derives a protected identity without cancelling the doubled part of the original polynomial.

\section{Four identities for occurrence-order alignment}\label{sec:alignment}

Let $\Pi$ be the set of the six permutations of $y,z,w$. Define
\begin{equation}\label{eq:guards}
 C_h=\sum_{\pi\in\Pi}2\pi yz,\qquad
 C_j=\sum_{\pi\in\Pi}2x\pi zyx.
\end{equation}
We call the words in these doubled sums \emph{guard words} and the sums \emph{guard polynomials}. They occur in the following two \emph{guarded identities}:
\begin{align}
 (H_\ell)\qquad yz+C_h&\approx zyz+C_h,\label{eq:Hl}\\
 (J_\ell)\qquad xyxzxyx+C_j&\approx xzxyx+C_j.\label{eq:Jl}
\end{align}
Reversing each multiplicative word gives the guarded identities $(H_r)$ and $(J_r)$. The four guarded identities involve at most four variables. Every guard polynomial in~\eqref{eq:guards} has exactly six displayed doubled summands.

\begin{lemma}\label{lem:extract-guards}
Let $\TT\in\mathsf W\subseteq\AA$, and suppose $\mathsf W\models P\approx Q$. For a common isolated content $D$, let $u,v$ be its occurrences on the two sides.
\begin{enumerate}[label=(\roman*),leftmargin=2.4em]
\item If $u,v$ have different first letters, then $\mathsf W\models(H_\ell)$.
\item If they have the same first letter but different first-occurrence words, then $\mathsf W\models(J_\ell)$.
\end{enumerate}
The corresponding assertions for last occurrences give $(H_r)$ and $(J_r)$.
\end{lemma}
\begin{proof}
For (i), let $a$ be the first letter of $v$. Replace $a$ by $z$, every other letter of $D$ by $y$, and every outside letter by $w$. Multiply the identity on the right by $yz$. Both isolated occurrences have content $\{y,z\}$; one begins with $y$, the other with $z$. A word over two generators in a band is determined by its content and its first and last letters. Deleting consecutive repeated letters leaves alternating words, and those of length at least four reduce using $(yz)^2\approx yz$ and $(zy)^2\approx zy$. The two isolated occurrences become $yz$ and $zyz$.

Each other occurrence contains an outside letter and, after multiplication by $yz$, has content $\{y,z,w\}$. It is nonisolated: the isolated occurrence and this occurrence both have support in that content. Equation~\eqref{eq:fourtwo} shows that the doubled transformed polynomial absorbs the double of this occurrence. Lemma~\ref{lem:lifting} permits replacing the occurrence by its double. Its last-occurrence word is $wyz$. By~\eqref{eq:rbword}, its double is one of the six terms $2\pi yz$ in $C_h$. Adding $C_h$ to the transformed identity absorbs all these doubled terms on both sides and yields~\eqref{eq:Hl}.

For (ii), let $p_0$ be the longest common prefix of $\first(u)$ and $\first(v)$, and let $a$ be the next letter in $\first(v)$. Put $D_0=\ct(p_0)$. Replace letters of $D_0$ by $x$, the letter $a$ by $Z=xzx$, all other letters of $D$ by $Y=xyx$, and all outside letters by $W=xwx$. Multiply on the left by $x$ and on the right by $ZY$.

The values $x,Y,Z,W$ lie in the local band $xSx$, with identity element $x$. Here $x$ is a nonempty term, not a constant symbol. After we omit this local identity from the calculation, the two isolated occurrences are words over $Y,Z$; one begins with $Y$, the other with $Z$, and both contain both letters. Appending $ZY$ gives the values $YZY$ and $ZY$. In the original variables, these are the undoubled words in~\eqref{eq:Jl}.

Every remaining occurrence has full content $\{x,y,z,w\}$, first letter $x$, and last-occurrence word $wzyx$. We double it by applying Lemma~\ref{lem:lifting} to the isolated occurrence and this occurrence. Each doubled word is one of $2x\pi zyx$. Adding $C_j$ yields~\eqref{eq:Jl}. Reversal gives the dual assertions.
\end{proof}

\begin{lemma}\label{lem:licensed-insertion}
Let $\mathsf V\subseteq\IB$, let $R$ have full content $X$, and let $u,v$ be occurrences of $R$ with the same content $D\subsetneq X$. Choose a word $g$ with $\ct(g)=X$ and $\mathsf V\models R+g\approx R$.
\begin{enumerate}[label=(\roman*),leftmargin=2.4em]
\item If $a$ is the first letter of $v$, instantiate $(H_\ell)$ by $y=u$, $z=a$, $w=g$, and multiply it on the right by $u$. Its undoubled words reduce to $u$ and $au$. Each of its six guard words has content $X$ and is absorbed by $R$ in $\mathsf V$.
\item Suppose $\first(u),\first(v)$ have a common nonempty prefix $p_0$, and $a$ is the next letter of $\first(v)$ after $p_0$. Choose an actual prefix $p$ of $u$ with $\ct(p)=\ct(p_0)$ and $\first(p)=p_0$. Instantiate $(J_\ell)$ by $x=p$, $y=u$, $z=a$, $w=g$, and multiply on the right by $u$. Its undoubled words reduce to $u$ and $pau$. Each of its six guard words has content $X$ and is absorbed by $R$ in $\mathsf V$.
\end{enumerate}
In (ii), such a prefix $p$ exists whenever $p_0$ is a proper common prefix of the first-occurrence words. The dual assertions hold for right insertion.
\end{lemma}
\begin{proof}
For (i), Lemma~\ref{lem:deletion} and band idempotence give
\[
 uau\approx u,\qquad auau\approx au.
\]
For (ii), the actual prefix relation gives $pu\approx u$. Substitution and right multiplication by $u$ give the two undoubled words
\[
 pupapupu\approx u,\qquad papupu\approx pau.
\]
For the first reduction, simplify $pu\approx u$ repeatedly and use $upau\approx u$ since $\ct(pa)\subseteq\ct(u)$. The second reduction follows by successive uses of $pu\approx u$.

To check absorption of the guard words, apply Lemma~\ref{lem:licensing}. In (i), their first-occurrence words come from the six orders of $u,a,g$ after deleting repetitions. In (ii), they come from
\begin{equation}\label{eq:sixprefixes}
 puag,\quad puga,\quad paug,\quad pagu,\quad pgua,\quad pgau.
\end{equation}
The first-occurrence words of $u$ and $v$ are admissible. Their prefixes $\first(pu)$ and $\first(pa)$ are admissible in (ii); in (i), the corresponding admissible prefixes are $\first(u)$ and $a$. Lemma~\ref{lem:admissible-concat} gives admissibility of $\first(pau)$ and permits completing each prefix with the admissible full word $\first(g)$. Placing $g$ earlier in the concatenation also preserves admissibility. This verifies all six words in~\eqref{eq:sixprefixes} and the six words in (i).

All six guard words in either case have the same last-occurrence word: delete the letters of $D$ from $\last(g)$ and append $\last(u)$. A full copy of $u$ follows every occurrence of $g$, while all other inserted factors have content in $D$. After reversal, retaining first occurrences in the concatenation $\rev(u)\rev(g)$ gives the corresponding first-occurrence word. The right-hand version of Lemma~\ref{lem:admissible-concat} proves its admissibility. The factor $g$ supplies full content. Lemma~\ref{lem:licensing} proves absorption of every guard word.
\end{proof}

Here $p$ must be an actual prefix of $u$. Its repetition-free first-occurrence word need not satisfy $pu\approx u$ in an arbitrary band.

\begin{lemma}\label{lem:selfguards}
Under the assumptions of Lemma~\ref{lem:licensed-insertion}, both $ugu$ and $vgv$ have content $X$ and are absorbed by $R$ in $\mathsf V$.
\end{lemma}
\begin{proof}
Deleting repetitions from $\first(u)\first(g)$ gives the first-occurrence word of $ugu$. The same operation on $\first(\rev(u))\first(\rev(g))$ gives its reversed last-occurrence word. Apply Lemmas~\ref{lem:admissible-concat} and~\ref{lem:licensing}. Replacing $u$ by $v$ gives the proof for $vgv$.
\end{proof}

\section{Finite bases and the main theorem}\label{sec:mainproof}

For a word identity $\sigma=(s\approx t)$ write
\[
 \sigma^{\sharp}:\quad s+2s\omega s\approx t+2t\omega t,
\]
where $\omega$ is fresh. For a polynomial identity $\tau=(P\approx Q)$ write $2\tau$ for $2P\approx2Q$. These are abbreviations for identities in the original two-operation signature.

\begin{theorem}\label{thm:finitebasisA}
Let $\TT\in\mathsf W\subseteq\AA$. Let $\mathsf U$ be the band variety generated by the multiplicative reducts of members of $\mathsf W$, and put $\mathsf V=\mathsf W\cap\IB$. Choose finite relative bases
\begin{itemize}[leftmargin=2.1em]
\item $\Sigma_{\mathsf U}$ for $\mathsf U$ inside the variety of bands;
\item $\Sigma_{\mathsf V}$ for $\mathsf V$ inside $\IB$;
\item $\Sigma_{\G}$ for the band theory $\G_{\mathsf W}$.
\end{itemize}
Let $\Gamma_{\mathsf W}$ consist of those members of $\{H_\ell,J_\ell,H_r,J_r\}$ that hold in $\mathsf W$. A finite identity basis for $\mathsf W$ is
\begin{equation}\label{eq:finitebasis}
\begin{split}
 \sr\;\cup\;\{x^2\approx x,\ 3x\approx2x\}
 &\;\cup\;\Sigma_{\mathsf U}
 \;\cup\;\{2\tau:\tau\in\Sigma_{\mathsf V}\}\\
 &\;\cup\;\Gamma_{\mathsf W}
 \;\cup\;\{\sigma^{\sharp}:\sigma\in\Sigma_{\G}\}.
\end{split}
\end{equation}
\end{theorem}
\begin{proof}
The finite bases exist by the finite basis theorems for bands~\cite{GP} and $\IB$~\cite{GPZ,Pastijn}, and by Lemma~\ref{lem:G}. Every identity in~\eqref{eq:finitebasis} holds in $\mathsf W$. Let $\mathsf Z$ be the variety they define. These identities give $\mathsf W\subseteq\mathsf Z\subseteq\AA$. To prove the reverse inclusion, we derive every identity of $\mathsf W$ in $\mathsf Z$.

Every doubled image of a member of $\mathsf Z$ belongs to $\mathsf V$. Since $d$ is a homomorphism, evaluating $P$ at doubled elements gives $2P$ at the original elements. The identities $2\Sigma_{\mathsf V}$ imply $P\approx Q$ on each doubled image whenever $2P\approx2Q$ occurs in the basis. The protected identities in the basis give
\begin{equation}\label{eq:Gcontain}
 \G_{\mathsf W}\subseteq\G_{\mathsf Z}.
\end{equation}
By Lemma~\ref{lem:G}, $\G_{\mathsf Z}$ is a fully invariant band theory. It contains $\Sigma_{\G}$ because their protected versions occur in~\eqref{eq:finitebasis}, and contains the band theory they generate.

Assume $\mathsf W\models P\approx Q$ and expand both terms as polynomials. Put $X=\ct(P)\cup\ct(Q)$. By Lemma~\ref{lem:count}, their count functions and isolated contents agree. If $\nu_P(X)=1$, both sides are single words, and their identity follows from $\Sigma_{\mathsf U}$. It remains to consider polynomials with at least two occurrences on each side.

\emph{The common doubled part.}
Let
\[
 \Dset=\Dset_{\mathsf V}(P)=\Dset_{\mathsf V}(Q),\qquad
 C=\sum_{q\in\Dset}2q.
\]
By Proposition~\ref{prop:normal}, the following identities hold in $\mathsf Z$:
\begin{equation}\label{eq:common-normal}
 P\approx C+\sum_{D\in\Iso}u_D,\qquad
 Q\approx C+\sum_{D\in\Iso}v_D,
\end{equation}
where $\Iso$ is the common set of isolated contents. The polynomial $C$ is nonempty. Every isolated content is proper in $X$, since otherwise all occurrences of its side would be supported in it.

Keep $C$ fixed. Starting from the first expression in~\eqref{eq:common-normal}, apply only transformations described below that we have already proved in $\mathsf Z$, and call the resulting polynomial $P_0$. Let $Q_0$ remain the second expression. The identity $P_0\approx Q_0$ holds in $\mathsf W$, since $\mathsf W\subseteq\mathsf Z$ and the original identity holds in $\mathsf W$. Put $R=P_0+Q_0$. In $\mathsf V$ we have $R\approx P\approx Q$. If a word $h$ has content $X$ and $\mathsf V\models R+h\approx R$, then
\begin{equation}\label{eq:available}
 h\in\Dset,\qquad C+2h\approx C.
\end{equation}
For the first statement, use the chosen free-band representative: $\nu_P(X)=2$, and $\mathsf V\models R\approx P$ gives the absorption condition in~\eqref{eq:Dset}. For the second, let $\bar h$ be the representative of $h$ in $B_X$. The sum $C$ contains $2\bar h$, and $h\approx\bar h$ holds in every band. Adding $2h$ changes $2\bar h$ to $4\bar h$, which equals $2\bar h$. This coefficient reduction proves $C+2h\approx C$ in $\mathsf Z$ using its semiring and band identities. The absorption test in $\mathsf V$ identifies a summand already in $C$; this argument does not assume $P_0\approx Q_0$ in $\mathsf Z$.

To choose $g$, multiply all word occurrences of $R$. The resulting word has content $X$, and $R$ absorbs it in $\mathsf V$ because it absorbs every factor and their product. The content remains $X$ at every stage: $C$ contains the word obtained initially by multiplying all occurrences of $P$ and $Q$, using $\mathsf V\models P\approx Q$.

\emph{Alignment of the first-occurrence words.}
Fix a common isolated content $D$ and write $u,v$ for its current occurrences in $P_0,Q_0$. Suppose their first-occurrence words differ. If their first letters differ, Lemma~\ref{lem:extract-guards} shows that $H_\ell\in\Gamma_{\mathsf W}$. Apply it as in Lemma~\ref{lem:licensed-insertion}(i), with $a$ the first letter of $v$. Its undoubled words are $u,au$. The polynomial $R$ absorbs each instantiated guard word, and each has content $X$. By~\eqref{eq:available}, the instantiated guard polynomial, say $C'$, satisfies $C+C'\approx C$. The guarded identity gives, in $\mathsf Z$,
\[
 C+u\approx C+C'+u\approx C+C'+au\approx C+au.
\]
We add the other isolated summands to both sides.

If the first letters agree, let $p_0$ be the longest common proper prefix of $\first(u),\first(v)$, and let $a$ be the next target letter from $\first(v)$. Lemma~\ref{lem:extract-guards} gives $J_\ell\in\Gamma_{\mathsf W}$. Choose an actual prefix $p$ as in Lemma~\ref{lem:licensed-insertion}(ii). That lemma and~\eqref{eq:available} allow the same calculation to replace $C+u$ by $C+pau$ in $\mathsf Z$.

Each replacement lengthens the common first-occurrence prefix with $v$. It preserves content $D$ and the last-occurrence word, since the former word $u$ remains as a suffix. No occurrence in the fixed doubled part is supported in an isolated $D$: a word $q\in\Dset$ with $\ct(q)\subseteq D$ would give $\nu_P(D)=2$. The replacement preserves isolation. Since the current identity $P_0\approx Q_0$ still holds in $\mathsf W$, every guarded identity needed at the next step belongs to the selected set $\Gamma_{\mathsf W}$. After finitely many steps the first-occurrence words agree.

\emph{Alignment of the last-occurrence words.}
Apply the reverse argument, using $H_r,J_r$. These replacements leave the whole current word as a prefix and preserve the first-occurrence word just obtained. After finitely many steps the resulting word $u^*$ satisfies
\begin{equation}\label{eq:aligned}
 \first(u^*)=\first(v),\qquad\last(u^*)=\last(v).
\end{equation}
These transformations hold in $\mathsf Z$ and retain the same $C$.

\emph{Replacement of the remaining word.}
The current identity $P_0\approx Q_0$ holds in $\mathsf W$, and its occurrences $u^*,v$ remain isolated with content $D$. Theorem~\ref{thm:extraction} and~\eqref{eq:aligned} imply $u^*\equiv_{\G_{\mathsf W}}v$. By~\eqref{eq:Gcontain}, the protected identity
\[
 u^*+2u^*\omega u^*\approx v+2v\omega v
\]
holds in $\mathsf Z$. Choose an absorbed full-content word $g$ for the current $R$. Lemma~\ref{lem:selfguards} and~\eqref{eq:available} show that $C$ absorbs both $2u^*gu^*$ and $2vgv$. Substituting $\omega=g$ and adding $C$ yields
\[
 C+u^*\approx C+u^*+2u^*gu^*
       \approx C+v+2vgv\approx C+v.
\]
This replaces the isolated occurrence without changing $C$ or the other isolated contents.

Repeating this replacement for each of the finitely many $D\in\Iso$ makes the expressions in~\eqref{eq:common-normal} identical; if $\Iso$ is empty, both already equal $C$. The completed derivation proves $\mathsf Z\models P\approx Q$, giving $\Id(\mathsf W)\subseteq\Id(\mathsf Z)$ and $\mathsf Z=\mathsf W$.
\end{proof}

\begin{corollary}\label{cor:A}
The variety $\AA$ is hereditarily finitely based.
\end{corollary}
\begin{proof}
A subvariety containing $\TT$ has the finite basis in Theorem~\ref{thm:finitebasisA}. A subvariety not containing $\TT$ is contained in $\JJ$ by Proposition~\ref{prop:forcedT}, and is finitely based by Proposition~\ref{prop:j}.
\end{proof}

\begin{proof}[Proof of Theorem~\ref{thm:main}]
The varieties $\RR$ and $\AA$ are hereditarily finitely based by Proposition~\ref{prop:r} and Corollary~\ref{cor:A}. One identity defines each relative to $\MI$. Proposition~\ref{prop:decompose} gives the required decomposition for every subvariety of $\MI$. Lemma~\ref{lem:compact} proves that $\MI$ is hereditarily finitely based. Corollary~\ref{cor:algebras} follows by applying this conclusion to $\Var(S)$.
\end{proof}

An algebra generating a locally finite variety is called \emph{inherently nonfinitely based} if it belongs to no finitely based locally finite variety of the same signature. The next corollary follows from the local finiteness of $\MI$ in Lemma~\ref{lem:localfinite}, independently of Theorem~\ref{thm:main}.

\begin{corollary}\label{cor:infb}
No multiplicatively idempotent semiring is inherently nonfinitely based in the constant-free signature $\{+,\cdot\}$.
\end{corollary}
\begin{proof}
The defining axioms give a finite basis for $\MI$, which is locally finite by Lemma~\ref{lem:localfinite}. Every multiplicatively idempotent semiring belongs to a finitely based locally finite variety of the same signature.
\end{proof}

\begin{remark}
Theorem~\ref{thm:main} proves that every subvariety of $\MI$ has a finite identity basis. The construction in~\eqref{eq:finitebasis} uses bases for the multiplicative band theory, the doubled image and the protected band theory, together with a subset of four fixed guarded identities. The auxiliary bases depend on the chosen subvariety. The proof gives no uniform bound on the number of variables, does not supply individual bases for every subvariety, and does not describe the full subvariety lattice of $\MI$.
\end{remark}
\section*{Declaration of AI use}
The authors used OpenAI ChatGPT and Codex for exploratory mathematical checks, manuscript
organization, and language editing. They reviewed and edited the resulting content and take
full responsibility for this article.

\end{document}